\documentclass[10pt]{amsart}
\usepackage{fancyhdr}
\usepackage{stmaryrd}
\usepackage{calrsfs}

\usepackage{amsmath}
\usepackage[nospace,noadjust]{cite}
\usepackage{amsfonts}
\usepackage{amssymb,enumerate}
\usepackage{amsthm}
\usepackage{cite}
\usepackage{comment}
\usepackage{color}
\usepackage[all]{xy}
\usepackage{hyperref}
\usepackage{lineno}
\usepackage{tikz-cd}

\newtheorem{theorem}{Theorem}[section]

\newtheorem{question}[theorem]{Question}

\newtheorem{lemma}[theorem]{Lemma}

\newtheorem*{main-theorem}{Main Theorem}

\theoremstyle{definition}

\newtheorem{example}[theorem]{Example}
\numberwithin{equation}{section}

\newcommand{\nn}{\mathbb{N}}
\newcommand{\pp}{\mathbb{P}}
\newcommand{\qq}{\mathbb{Q}}
\newcommand{\rr}{\mathbb{R}}
\newcommand{\zz}{\mathbb{Z}}

\newcommand{\gp}{\mathcal{G}}

\newcommand{\uu}{\mathcal{U}}

\providecommand\ldb{\llbracket}
\providecommand\rdb{\rrbracket}

\hypersetup{
	pdftitle={FD},
	colorlinks=true,
	linkcolor=blue,
	citecolor=cyan,
	urlcolor=wine
}



\begin{document}
	
	\mbox{}
    \title{On near atomicity and a characterization of the FF property}

	\author{Jonathan Du}
	\address{Department of Mathematics\\MIT\\Cambridge, MA 02139}
	\email{jndu@mit.edu}

	\author{Felix Gotti}
	\address{Department of Mathematics\\MIT\\Cambridge, MA 02139}
	\email{fgotti@mit.edu}
	
	\author{Leo Hong}
	\address{MIT PRIMES\\MIT\\Cambridge, MA 02139}
	\email{lhong6@charlotte.edu}

\date{\today}

\begin{abstract}	 
	A commutative cancellative monoid is atomic if every nonunit factors into atoms, and an integral domain is atomic if its multiplicative monoid of nonzero elements is atomic. Several weakenings of atomicity have been introduced and studied during the past decade, including near atomicity, almost atomicity, and quasi-atomicity. Although nearly atomic monoids that are not atomic were already known, whether there exist nearly atomic integral domains that are not atomic had remained open. We answer this question affirmatively by constructing an explicit nearly atomic integral domain that is not atomic. We also strengthen the classical Anderson--Anderson--Zafrullah characterization of the finite factorization property by proving that an integral domain is an FFD if and only if it is both nearly atomic and IDF. We conclude by showing that near atomicity cannot be weakened to almost atomicity in this characterization, even within the class of IDF domains.
\end{abstract}
\medskip

\maketitle

%\tableofcontents

\bigskip
%%%%%%%%%%%
%%%%%%%%%%%
\section{Introduction}
\label{sec:intro}
\smallskip

The Fundamental Theorem of Arithmetic (FTA) states that both existence and uniqueness of factorizations are guaranteed in the prototypical integral domain $\zz$. Motivated by this theorem, UFDs and related factorization properties have been systematically investigated for more than a century. The term atomicity refers to the existential condition in the FTA: an integral domain is atomic if every nonzero nonunit can be factored into atoms (also called irreducibles). Atomicity was studied by Grams~\cite{aG74} and Zaks~\cite{aZ82} in connection with ascending chains of principal ideals (see~\cite{GL23,GL23a,BGLZ24} for more recent work in the same direction). The ascent of atomicity to polynomial rings was settled in~\cite{mR93}, while its ascent to monoid algebras was settled in~\cite{CG19} (see also \cite{GR25}). The first systematic study of atomicity in the setting of integral domains was carried out by Anderson, Anderson, and Zafrullah in~\cite{AAZ90}, where the authors introduced and investigated the bounded factorization (BF) and the finite factorization (FF) properties, two natural generalizations of the unique factorization (UF) property, which has received a great deal of attention since then.
\smallskip

In this paper, we study a weaker notion of atomicity, known as near atomicity, and then we connect it to the FF property. Following Lebowitz-Lockard~\cite{nL19}, we say that an integral domain $R$ is nearly atomic provided that there exists a nonzero $s \in R$ such that every nonzero element in the principal ideal $sR$ is atomic in $R$ (i.e., factors into irreducibles). It follows directly from the definitions that every atomic domain is nearly atomic. In~\cite{nL19}, Lebowitz-Lockard studies near atomicity in connection with further subatomic properties, including the quasi-atomic and the almost atomic properties, both introduced by Boynton and Coykendall in~\cite{BC15} in their study of the graph of divisibility of an integral domain. Near atomicity was subsequently considered in the more general setting of monoids by Gotti and Vulakh~\cite{GV23} and then by Liu, Rodriguez, and Tirador~\cite{LRT24}.
\smallskip

Following~\cite{BC15}, we say that an integral domain $R$ is quasi-atomic (resp., almost atomic) provided that for each nonzero element $r \in R$ there exists an element (resp., an atomic element) $s \in R$ such that $rs$ is atomic in $R$. The ascent of almost and quasi-atomicity to polynomial semidomains was first studied by Polo and the second author in~\cite{GP23}, and the ascent of the same two properties to the more general case of monoid semidomains was more recently studied by Gonzalez, Panpaliya, and the second author~\cite{GGP25}. It is not difficult to prove that every nearly atomic domain is almost atomic, while it follows directly from the definitions that every almost atomic domain is quasi-atomic. For more information about almost and quasi-atomicity in the setting of integral domains, the interested reader can consult Section~10 of the recent survey \cite{CG24}.
\smallskip

The first main goal of this paper is to fill a gap in the literature by providing an explicit construction of a nearly atomic domain that is not atomic. The idea is to build the domain in stages inside a larger ambient group algebra. At each stage, we add carefully chosen elements so that multiplying by one fixed monomial forces the elements already constructed to factor into irreducibles, while a separate family of monomials remains outside all such factorizations and prevents the whole domain from being atomic. This construction shows that near atomicity is genuinely weaker than atomicity for integral domains. Although near atomicity was studied for integral domains~\cite{nL19} and for submonoids of rank-$2$ abelian groups~\cite{LRT24}, the question of whether there exist integral domains that are nearly atomic but not atomic had remained open. The implications among the atomic properties relevant to this paper is shown Figure~\ref{fig:atomicity diagram}.
\begin{center}
	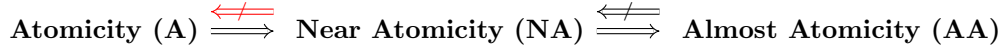
\begin{figure}[h]
		\begin{tikzcd} %[cramped]
			\textbf{ Atomicity (A)}   \arrow[r, Rightarrow] \arrow[red, r, Leftarrow, "/"{anchor=center,sloped}, shift left=1.7ex] 
			& \textbf{ Near Atomicity (NA)} \arrow[r, Rightarrow] \arrow[r, Leftarrow, "/"{anchor=center,sloped}, shift left=1.7ex] 
			& \textbf{ Almost Atomicity (AA)}
		\end{tikzcd}
		\caption{The diagram shows the implications among atomicity and the two weaker properties considered here: near atomicity and almost atomicity. The first marked reverse implication (in red) is one of the two main results of this paper.}
		\label{fig:atomicity diagram}
	\end{figure}
\end{center}

We say that an integral domain (resp., a commutative monoid) has the IDF property provided that every nonzero element of the domain (resp., element of the monoid) has only finitely many irreducible divisors up to associates. The IDF property was first studied by Grams and Warner~\cite{GW75} in 1975. The ascent of the IDF property to polynomial extensions was one of the two open problems left in~\cite{AAZ90} in 1990 (it was answered negatively by Malcolmson and Okoh \cite[Theorem~2.5]{MO09} in 2009). Further work on the ascent of the IDF property to polynomial extensions has been carried out by Eftekhari and Khorsandi~\cite{EK18}, by Gotti and Zafrullah~\cite{GZ23} and, more recently, by Gonzalez and Panpaliya~\cite{GP25}. 
\smallskip

The other main result of this paper is a characterization of the finite factorization (FF) property in terms of near atomicity. The FF property was introduced and investigated in the first systematic study of factorizations in integral domains by Anderson, Anderson, and Zafrullah~\cite{AAZ90}. In the same paper, the authors proved that, for an integral domain, being both atomic and IDF is equivalent to satisfying the FF property. The FF property, along with this characterization, was extended to the larger class of cancellative commutative monoids by Halter-Koch~\cite{fHK92}. The FF property has been the subject of various papers in recent literature, including~\cite{BHL17,fG19,DG25,JKK24}. Here we sharpen the Anderson--Anderson--Zafrullah characterization by proving that every nearly atomic IDF domain satisfies the FF property. Thus, in the presence of the IDF property, near atomicity already forces atomicity and finite factorization.
Figure~\ref{fig:atomicity+IDF diagram} records this improvement and also indicates the limit of the result: replacing near atomicity by almost atomicity is not sufficient in general.
\begin{center}
	\begin{figure}[h]
		\begin{tikzcd} [cramped] 
			\textbf{ FF }    \arrow[r, Leftrightarrow] 
			&\textbf{ [A + IDF] } \arrow[r, Rightarrow] \arrow[red, r, Leftarrow, shift left=1.7ex] 
			& \textbf{ [NA + IDF] } \arrow[r, Rightarrow] \arrow[red, r, Leftarrow, "/"{anchor=center,sloped}, shift left=1.7ex] 
			& \textbf{ [AA + IDF] }
		\end{tikzcd}
		\caption{The diagram shows the implications obtained by combining the IDF property with atomicity, near atomicity, and almost atomicity. The first marked reverse implication (in red) is one of the two main results of this paper, while the second marked reverse implication fails in general.}
		\label{fig:atomicity+IDF diagram}
	\end{figure}
\end{center}

This paper is organized as follows. In Section~\ref{sec:background}, we review the notation, terminology, and established results needed to make the paper as self-contained as possible. In Section~\ref{sec:non-atomic nearly atomic domain}, we construct a nearly atomic domain that is not atomic. Finally, in Section~\ref{sec:a characterization of FFDs via near atomicity}, we characterize finite factorization domains as nearly atomic IDF domains, and we conclude with an example illustrating that an almost atomic IDF domain need not satisfy the FF property.

\bigskip
%%%%%%%%%%%
%%%%%%%%%%%
\section{Background}
\label{sec:background}

\smallskip
%%%%%%%%%%%%%%%%
\subsection{General Notation}
\smallskip

Let $\zz$, $\qq$, and $\rr$ be the sets of integers, rational numbers, and real numbers, respectively. In addition, let $\pp$, $\nn$, and $\nn_0$ be the sets of primes, positive integers, and nonnegative integers, respectively. For $b,c \in \zz$ with $b \le c$, we let $\ldb b, c \rdb$ be the discrete interval from $b$ to $c$:
\[
	\ldb b,c \rdb := \{m \in \zz : b \le m \le c\}.
\]
For a set $R$ consisting of real numbers and $s \in \rr$, we set $R_{\ge s} := \{r \in R : r \ge s\}$, and we define $R_{> s}$ similarly. Let $A$ be an abelian group, and let $S$ be a set. We let $A^S$ denote the set of all functions $S \to A$ with domain $S$ and codomain $A$, which is an abelian group under standard addition of functions. In addition, we define the \emph{support} of $\alpha \in A^S$ to be
\[
	\text{supp} \, \alpha := \{s \in S : \alpha(s) \neq 0 \},
\]
and we say that $\alpha$ is \emph{finitely supported} provided that $|\text{supp} \, \alpha| < \infty$. The subset consisting of all finitely supported functions of $A^S$ is a subgroup of $A^S$, which is denoted by $A^{(S)}$. When $S = \nn$, we often write a function $\alpha \colon S \to A$ in $A^{S}$ as a sequence $(\alpha_n)_{n \ge 1}$, where $\alpha_n := \alpha(n) \in A$ for every $n \in \nn$. Accordingly, the sequences in $A^{(\nn)}$ are called \emph{finitely supported} sequences.

\medskip
%%%%%%%%%%%%%%%%%%%
\subsection{Commutative Monoids}

Throughout this paper, the term \emph{monoid} always refers to a cancellative and commutative semigroup with an identity element. Unless otherwise specified, monoids are written multiplicatively; however, submonoids of totally ordered groups such as $\rr$ and $\qq$ will be written additively. A submonoid of $\rr_{\ge 0}$ is called a \emph{positive monoid}. Let $M$ be a monoid. The group of invertible elements of $M$, also called units, is denoted by $M^\times$ (or $\uu(M)$). The monoid $M$ is said to be \emph{reduced} if~$M^\times$ is the trivial group. 
\smallskip

The \emph{Grothendieck group} of $M$ is the abelian group $\gp(M)$, unique up to isomorphism, that satisfies the following property: any abelian group containing an isomorphic image of $M$ also contains an isomorphic image of $\gp(M)$. If $\gp(M)$ is a torsion-free group then we say that $M$ is a \emph{torsion-free} monoid. The \emph{rank} of $M$ is the rank of $\gp(M)$ as a~$\zz$-module or, equivalently, the dimension of the vector space $\qq \otimes_\zz \gp(M)$ over~$\qq$. If~$S$ is a subset of~$M$ then we let $\langle S \rangle$ denote the smallest submonoid of~$M$ containing~$S$.
\smallskip

For $b,c \in M$, we say that $c$ \emph{divides} $b$ if there exists $d \in M$ such that $b = cd$, in which case we write $c \mid_M b$. If for some $b,c \in M$ both relations $b \mid_M c$ and $c \mid_M b$ hold then $b$ and $c$ are called \emph{associates}. A submonoid $N$ of~$M$ is said to be \emph{divisor-closed} if $a \in N$ and $b \mid_M a$ together imply $b \in N$.
\smallskip

An element $a \in M \setminus M^\times$ is called an \emph{atom} provided that for any elements $b,c \in M$ the equality $a = bc$ guarantees that $\{b,c\}$ has nonempty intersection with $M^\times$. We let $\mathcal{A}(M)$ denote the set of atoms of $M$. We say that $M$ is an \emph{irreducible-divisor-finite} (IDF) monoid, or that $M$ has the \emph{IDF property}, if every element of $M$ is divisible only by finitely many atoms of $M$ up to associates. The IDF property was introduced and first studied by Grams and Warner~\cite{GW75} in 1975.
\smallskip

We say that an element $a \in M$ is \emph{atomic} if either $a$ belongs to $M^\times$ or $a$ factors into finitely many atoms. Following Cohn~\cite{pC68}, we say that the monoid $M$ is atomic if every element of $M$ is atomic. A monoid with no atoms is called \emph{antimatter}. The monoid $M$ satisfies the \emph{ascending chain condition on principal ideals} (ACCP) if every ascending chain of principal ideals of $M$ eventually stabilizes; equivalently, $M$ contains no infinite sequence $(a_n)_{n \ge 1}$ of pairwise nonassociate elements such that $a_{n+1} \mid_M a_n$ for every $n \in \nn$. It is well known that every monoid that satisfies the ACCP is atomic. We proceed to introduce two atomicity conditions weaker than atomicity. For $s \in M$, the set $sM := \{sm : m \in M\}$ is the principal ideal generated by $s$. We say that $M$ is \emph{nearly atomic} if there exists $s \in M$ such that every element of $sM$ is atomic in $M$. On the other hand, we say that $M$ is \emph{almost atomic} if for every $b \in M$ there exists an atomic element $a \in M$ such that $ab$ is also atomic in $M$. It follows from the definitions that every atomic monoid is nearly atomic, and one can verify that every nearly atomic monoid is almost atomic.
\smallskip

Let $\mathsf{Z}(M)$ denote the free commutative monoid on the set $\mathcal{A}(M/M^\times)$, and let $\pi \colon \mathsf{Z}(M) \to M/M^\times$ denote the unique monoid homomorphism that fixes every element of the set $\mathcal{A}(M/M^\times)$. The elements of $\mathsf{Z}(M)$ are called \emph{factorizations}. For each $m \in M$, we set
\[
	\mathsf{Z}(m) := \pi^{-1} (m M^\times),
\]
and we say that $M$ is a \emph{unique factorization monoid} (UFM) if $|\mathsf{Z}(m)| = 1$ for all $m \in M$. More generally, we say that $M$ is a \emph{finite factorization monoid} (FFM) if $1 \le |\mathsf{Z}(m)| < \infty$ for all $m \in M$. Observe that every UFM is an FFM. If $M$ is a UFM (resp., FFM) then we often say that $M$ has the \emph{UF property} (resp., \emph{FF property}). Interested readers can consult~\cite{AG22} for a recent survey on the FF property (and the BF property) in the setting of integral domains.

\medskip
%%%%%%%%%%%%%%%%%%%%%%%%%%%%%%%%%%%%%%%%%%%%%%%%%
\subsection{Integral Domains and Monoid Algebras} 

Let $R$ be an integral domain. We denote the multiplicative monoid of~$R$ by $R^*$ and the group of units of $R$ by $R^\times$. We say that $R$ is an \emph{atomic} (resp., a \emph{nearly atomic}, an \emph{almost atomic}, an \emph{IDF}, an \emph{FF}) domain provided that $R^*$ is an atomic (resp., a nearly atomic, an almost atomic, an IDF, an FF) monoid.
\smallskip

Let $M$ be a monoid, and let $R$ be a commutative ring with identity element~$1$. Then the abelian group $R^{(M)}$ consisting of all finitely supported functions $M \to R$ is a free $R$-module on the set $\{\delta(m,-) : m \in M\}$, where $\delta \colon M \times M \to \{0,1\}$ is the Kronecker delta function on~$M$. It is convenient to represent the functions in $R^{(M)}$ in a polynomial-like fashion: let $x$ be an indeterminate, and set $x^m := \delta(m,-) : M \to R$, which is the function sending $m \mapsto 1$ and $m' \mapsto 0$ for all $m' \neq m$. Then $R^{(M)}$, also denoted by $R[x;M]$, is the free $R$-module with basis $\{x^m : m \in M \}$. We can and will consider $R[x;M]$ as an $R$-algebra under the multiplication determined by the relations $x^{m_1}x^{m_2} = x^{m_1+m_2}$ (for all $m_1, m_2 \in M$) extended according to the distributive law. In this case, we call the elements of $R[x;M]$ \emph{polynomial expressions}. When considered as an $R$-algebra, $R[x;M]$ is called the \emph{monoid algebra} of $M$ over $R$. Moreover, if $M$ is an abelian group then $R[x;M]$ is called the \emph{group algebra} of~$M$ over~$R$. The monoid algebra $R[x;M]$ is an integral domain if and only if~$M$ is torsion-free (besides cancellative and commutative) and $R$ is an integral domain \cite[Theorem~8.1]{rG84}; in this case, $R[x;M]$ is also called the \emph{monoid domain} of $M$ over $R$.
\smallskip

Since this paper only considers monoid algebras that are integral domains, let us fix an integral domain $R$ and a torsion-free monoid $M$ so that the monoid algebra $R[x;M]$ is an integral domain. Following Gilmer's notation, we denote $R[x;M]$ simply by $R[M]$. As any nonzero polynomial expression $f \in R[M]$ has nonempty support, namely, $\text{supp} \, f = \{s_1, \dots, s_n\}$ for some pairwise distinct $s_1, \dots, s_n \in M$, we can write $f$ as follows:
\begin{equation} \label{eq:generic element in a monoid algebra}
	f := c_1 x^{s_1} + \dots + c_n x^{s_n}
\end{equation}
for some unique nonzero coefficients $c_1, \dots, c_n \in R$. If $\text{supp} \, f$ is a singleton then we refer to $f$ as a \emph{monomial}. It is well known that the group of units of $R[M]$ is the set of invertible monomials of $R[M]$:
\[
	R[M]^\times = \{ux^e : (u,e) \in R^\times \times \uu(M) \}.	
\]
It was established by Levi~\cite{fL13} that an abelian group $A$ admits a total order compatible with its operation if and only if $A$ is torsion-free. As a consequence of this result, it follows that a (cancellative and commutative) monoid admits a compatible total order $\preceq$ if and only if it is torsion-free. Hence we can assume that $M$ is endowed with a total order compatible with its monoid operation in the sense that order relations are preserved under the addition operation: for all $a,b,c \in M$, if $a \prec b$ then $a+c \prec b+c$. In this case, we can extend the notions of order and degree from standard polynomial rings to monoid algebras:
\[
	\text{ord} \, f :=  \min \, \text{supp} \, f \quad  \text{ and } \quad \deg f := \max \, \text{supp} \, f
\]
are called the \emph{order} and \emph{degree} of $f$, respectively. In addition, for all nonzero $f,g \in R[M]$, the following identities hold: $\text{ord} \, fg = (\text{ord} \, f) + (\text{ord} \, g)$ and $\deg fg = (\deg f) + (\deg g)$.
\smallskip

%When $R[M]$ is an integral domain, it follows from~\cite[Proposition 8.3]{GP74} that $\dim R[M] \ge 1 + \dim R$. In addition, when $R$ is Noetherian, it follows from \cite[Corollary~2]{jO88} that $\dim R[M] = \dim R + \text{rank} \, M$. As a result, the one-dimensional monoid algebras that are integral domains are precisely the monoid algebras of rank-one torsion-free monoids over fields. Background information on monoid algebras $R[M]$, emphasizing on the ascent of algebraic properties from the pair $(M, R)$ to $R[M]$ and including the most significant progress on the same subject until 1984, can be found in Gilmer's book~\cite{rG84}.

%\smallskip
%%%%%%%%%%%%%%%%%%%%%%%%%%%%
%\section{Classes of Non-Atomic Nearly Atomic Domains}
%
%This first section of content is devoted to the construction of the first known class of nearly atomic domains that are not atomic. These integral domains are obtained inside a group algebra by taking the nested union of certain subrings. %The second will be obtained by localizing a monoid domain, following the classical Grams' construction of the first atomic domain not satisfying the ACCP. 

\bigskip
%%%%%%%%%%%%%%%%%%%%%%%%%%%%%%%%%%
%%%%%%%%%%%%%%%%%%%%%%%%%%%%%%%%%%
\section{A Class of Nearly Atomic Domains that Are not Atomic}
\label{sec:non-atomic nearly atomic domain}

This first section of content is devoted to the construction of the first known class of nearly atomic domains that are not atomic. These integral domains are obtained inside a group algebra by taking the nested union of certain subrings. First, we need to construct a class of auxiliary non-atomic positive monoids, which will play an essential role later.

\medskip
%%%%%%%%%%%%%%%%%%%%%%%%%%%%%%%%%%%%%%%%%%%%%%%%%%%
\subsection{A Class of Non-Atomic Positive Monoids}

Fix a non-algebraic $r \in \rr_{> 0}$, and then let $A$ be a countable subset of $\rr$ such that $A \cup \{r\}$ is algebraically independent over $\qq$. We let $\nn_0 A$ denote the additive submonoid of $\rr$ generated by $A$ (that is, $\nn_0 A = \langle A \rangle$), and we let $\zz A$ denote the Grothendieck group of $\nn_0 A$ inside~$\rr$. Since $A$ is algebraically independent over $\qq$, the monoid $\nn_0 A$ and the abelian group $\zz A$ are both free.
\smallskip

Fix $\epsilon \in \rr_{> 0}$ such that $\epsilon < r/4$. After replacing $a$ by a suitable element of $\qq a$ (for all $a \in A$), we can further assume that
\begin{equation} \label{eq:subset A}
	A \subset \Big(\frac{r}2 - \epsilon, \frac{r}2 + \epsilon \Big).
\end{equation}
Let $e \colon \nn_0 \times \nn \to A \cup \{0\}$ be a function such that $e(0,k) = 0$ for every $k \in \nn$ and whose restriction $e \colon \nn \times \nn \to A$ is a bijection, and then consider the positive monoid
\begin{equation} \label{eq:subset B}
	E := \nn_0 A + \nn_0 B, \quad \text{where} \quad B := \big\{r - e(n,k) : (n,k) \in \nn^2 \big\}.
\end{equation}

\begin{lemma} \label{lem:Claim 1}
	Let the sets $A$ and $B$ and the monoid $E$ be as in \eqref{eq:subset A} and~\eqref{eq:subset B}. Then $E$ is an atomic monoid with $\mathcal{A}(E) = A \cup B$.
\end{lemma}

\begin{proof}
	Since $E$ is generated by $A \cup B$, proving that $E$ is atomic with $\mathcal{A}(E) = A \cup B$ amounts to showing that every element of $A \cup B$ is an atom of $E$. First, we show that $A \subseteq \mathcal{A}(E)$. Fix $a \in A$, and suppose that
	\[
		a =  \sum_{i=1}^m a_i + \sum_{i=1}^n (r - a'_i)
	\]
	for some $m,n \in \nn_0$ and $a_1, \dots, a_m, a'_1, \dots, a'_n \in A$. Then
	\[
		nr + \sum_{i=1}^m a_i - \sum_{i=1}^n a'_i - a = 0.
	\]
	The algebraic independence of $A \cup \{r\}$ over $\qq$ implies, in particular, that this set is linearly independent over $\qq$. Thus $n=0$. It follows that $a = \sum_{i=1}^m a_i$ inside the free monoid $\nn_0 A$, and therefore $m=1$ and $a_1=a$. Hence $a$ cannot be written as a sum of two nonzero elements of $E$, so $a \in \mathcal{A}(E)$. Thus, $A \subseteq \mathcal{A}(E)$.

	Now we show that $B \subseteq \mathcal{A}(E)$. Fix $r - a' \in B$, and suppose that
	\[
		r-a' =  \sum_{i=1}^m a_i + \sum_{i=1}^n (r - a'_i)
	\]
	for some $m,n \in \nn_0$ and $a_1, \dots, a_m, a'_1, \dots, a'_n \in A$. Then
	\[
		(n-1)r + \sum_{i=1}^m a_i - \sum_{i=1}^n a'_i + a' = 0.
	\]
	Again using the linear independence of $A \cup \{r\}$ over $\qq$, we obtain $n=1$. Hence
	\[
		a'_1 = a' + \sum_{i=1}^m a_i.
	\]
	Since $a'_1 \in A \subseteq \mathcal{A}(E)$, the last equality forces $\sum_{i=1}^m a_i = 0$ and $a'_1=a'$. Therefore the original expression of $r-a'$ is trivial, and so $r-a' \in \mathcal{A}(E)$. This proves that $B \subseteq \mathcal{A}(E)$, completing the proof.
\end{proof}

Let $Q$ denote $\nn_0\big[1/2]$, which is the positive monoid consisting of all nonnegative dyadic rationals, and observe that $Q$ is a valuation monoid: indeed, for each $q_1, q_2 \in Q$ the divisibility relation $q_1 \mid_Q q_2$ holds if and only if $q_1 \le q_2$. Now consider the positive monoids
\[
    E_0 := \nn_0 r \quad \text{ and } \quad E_n := \big\langle e(n,k), r - e(n,k) : k \in \nn \big\rangle
\]
for every $n \in \nn$. Observe that $E = \sum_{n \in \nn} E_n$. In addition, for each $n \in \nn$, the inclusion $E_0 \subset E_n$ holds because $r = e(n,1) + (r - e(n,1)) \in E_n$. Now define
\begin{equation} \label{eq:def of the monoids M and M_n}
	M := \bigcup_{n \in \nn_0} M_n, \quad \text{ where } \quad M_n := Q + \sum_{j=0}^n E_j
\end{equation}
for every $n \in \nn_0$. It is clear that $M_0 = Q + E_0$ and $M_n = M_{n-1} + E_n$ for every $n \in \nn$, whence $(M_n)_{n \in \nn_0}$ is an ascending chain of positive monoids, and so $M$ is a positive monoid.
Let us take a look at the structure of $M$.

\begin{lemma} \label{lem:M=Q+E=<A U B>}
	For the monoid $M$ in~\eqref{eq:def of the monoids M and M_n}, the following statements hold.
	\begin{enumerate}
		\item $M = Q \oplus E$, and so $Q$ is a divisor-closed submonoid of $M$.
		\smallskip
		
		\item $M$ is a reduced monoid with $\mathcal{A}(M) = A \cup B$, and so $M$ is not atomic.
	\end{enumerate}
\end{lemma}

\begin{proof}
	(1) As $E = \sum_{n \in \nn} E_n$, and as $E_0 \subset E$, it follows from~\eqref{eq:def of the monoids M and M_n} that
	\begin{equation} \label{eq:M in terms of A cup B}
		M = \bigcup_{n \in \nn_0} \bigg( Q + \sum_{j=0}^n E_j \bigg) = Q + \sum_{n \in \nn} E_n = Q+E.
	\end{equation}
	Observe that $E = \nn_0 A + \nn_0(r - A) \subset \qq A + \qq r$. Since $A \cup \{ r \}$ is algebraically independent over $\qq$, it is linearly independent over $\qq$. Thus no nonzero rational number can belong to $\qq A + \qq r$, and so $Q \cap E = \{0\}$. It follows that each element of $M$ has a unique expression as $q+e$ with $(q,e) \in Q \times E$. Hence the sum of the monoids $Q$ and $E$ is direct, so $M = Q \oplus E$. As a result, $Q$ is divisor-closed in $M$: if $q \in Q$ and $q=m+m'$ for some $m,m' \in M$ then the $E$-components of $m$ and $m'$ must both be zero, and so $m,m' \in Q$.
	\smallskip
	
	(2) As $M = Q \oplus E$, it is enough to verify that both $Q$ and $E$ are reduced. The monoid $Q$ is reduced because it is a Puiseux monoid, that is, a submonoid of $\qq_{\ge 0}$. The monoid $E$ is also reduced because every element of $A \cup B$ is positive, and so no two nonzero elements of $E$ can sum to $0$. Therefore $M$ is reduced.
	
	It follows from Lemma~\ref{lem:Claim 1} that $E = \langle A \cup B\rangle$. The inclusion $E_0 \subset E$, along with~\eqref{eq:M in terms of A cup B}, implies that $M = Q + E = M_0 + \langle A \cup B \rangle$. Thus, $M$ is generated by $M_0 \cup A \cup B$, and because $M$ is reduced, every atom of $M$ must belong to this generating set. We next show that no element of $M_0 = Q + \nn_0 r$ is an atom of $M$. Indeed, take $q+nr \in M_0$ with $q \in Q$ and $n \in \nn_0$. If $q>0$ then
	\[
		q+nr = \frac{q}{2} + \bigg(\frac{q}{2} + nr\bigg)
	\]
	is a decomposition into two nonzero elements of $M$. If $q=0$ and $n>0$, then $nr$ is not an atom: when $n=1$, we have $r=e(1,1)+(r-e(1,1))$, and when $n>1$, we have $nr=r+(n-1)r$. Hence $\mathcal{A}(M) \subseteq A \cup B$.

	For the reverse inclusion, fix $a \in A \cup B$, and write $a = c+d$ for some $c,d \in M  = Q \oplus E$. Since $a \in A \cup B \subset E$, comparing $Q$-components in the direct sum decomposition shows that $c,d \in E$. As $a \in \mathcal{A}(E)$ by Lemma~\ref{lem:Claim 1}, it follows that $\{c,d\} = \{0,a\}$. Hence $A \cup B \subseteq \mathcal{A}(M)$, and therefore $\mathcal{A}(M) = A \cup B$. 
	
	Finally, since $M = Q \oplus E = Q \oplus \langle \mathcal{A}(M) \rangle$, no nonzero element of $Q$ can be a sum of atoms of $M$. Because $Q$ contains nonzero elements, $M$ is not atomic.
\end{proof}

\medskip
%%%%%%%%%%%%%%%%%%%%%%%%%%%%%%%%%%%%%%%%%%%%%%%%%%%%%%%%%%%%%%
\subsection{Nested Sequence of Subrings Inside a Group Algebra}

Recall that $\zz^{(\nn)}$ denotes the abelian group consisting of all the finitely supported integer-valued sequences $(\gamma_k)_{k \ge 1}$. Let the elements of the set $A \cup \{r\}$ be algebraically independent positive real numbers as constructed in the previous section. Also, let the monoid $Q$ and the sequence of monoids $(E_n)_{n \ge 0}$ be defined as in the previous section. For this section, it is convenient to set $N := Q \oplus E$, where $E = \sum_{n \in \nn} E_n$, and
\[
	A = \bigcup_{n \in \nn} A_n \quad \text{ where } \quad A_n := \{e(n,k) : k \in \nn \}
\]
for every $n \in \nn$. Moreover, define $V_{-1} := \{0\}$ and
\[
	V := \bigcup_{n \in \nn_0} V_{\le n}, \quad \text{ where } \quad V_{\le n} := \qq \oplus \qq r \oplus \sum_{j=1}^n \qq A_j
\]
for every $n \in \nn_0$. From the fact that $A \cup \{r\}$ is algebraically independent over $\qq$, we obtain that $V_{\le n}$ (resp., $V$) is a vector space over $\qq$ with basis $\{1,r\} \cup \bigcup_{j=1}^n A_j$ (resp., $\{1,r\} \cup A$). Observe that $V_{\le 0} = \qq \oplus \qq r$. Finally, for every subring $S$ of $\zz[\rr]$, we let $\deg S$ denote the additive submonoid of $\rr$ consisting of the $\zz[\rr]$-degrees of all nonzero polynomial expressions of $\zz[\rr]$ that belong to~$S$. 

We shall use the following elementary degree-control observation throughout the construction. Suppose that, for some $n \in \nn_0$, we have a countable subring $R_n$ of $\zz[N]$ such that $x^r \in R_n$, every exponent occurring in a nonzero element of $R_n$ belongs to $V_{\le n}$, and $M_n := \deg R_n$. Let $(f_{n,k})_{k \ge 1}$ be a sequence whose underlying set is $R_n \setminus \{0\}$, and set
\[
	U_k := x^{e(n+1,k)} \quad \text{and} \quad W_k := f_{n,k}x^{r-e(n+1,k)}
\]
for every $k \in \nn$. If
\[
	R_{n+1} := R_n[U_k,W_k : k \in \nn]
\]
and $d(n+1,k) := \deg f_{n,k} + r - e(n+1,k)$, then
\begin{equation} \label{eq:degree-control-step}
	\deg R_{n+1} = M_n + \big\langle e(n+1,k), d(n+1,k) : k \in \nn \big\rangle.
\end{equation}
Indeed, every polynomial expression of $R_{n+1}$ is a finite sum of terms of the form
\[
	h \prod_{k \in F} U_k^{\alpha_k} W_k^{\beta_k},
\]
where $F \subset \nn$ is finite, $h \in R_n$, and $\alpha_k,\beta_k \in \nn_0$. Since $U_kW_k = x^r f_{n,k} \in R_n$, we can repeatedly absorb common powers of $U_k$ and $W_k$ into the coefficient $h$. Thus, every polynomial expression of $R_{n+1}$ can be written as a finite sum
\[
	\sum_{\gamma \in \zz^{(\nn)}} h_\gamma \prod_{k \in \nn} U_k^{\gamma_k^+} W_k^{\gamma_k^-},
\]
where $h_\gamma \in R_n$, $\gamma=(\gamma_k)$ runs over finitely many elements of $\zz^{(\nn)}$, and $\gamma_k^+ := \max\{\gamma_k,0\}$ and $\gamma_k^- := \max\{-\gamma_k,0\}$. For distinct $\gamma$, the corresponding summands have disjoint supports: the coefficient of $e(n+1,k)$ in every exponent occurring in the $\gamma$-summand is exactly $\gamma_k$. Hence cancellation cannot occur between different $\gamma$-summands. The degree of a nonzero $\gamma$-summand is
\[
	\deg h_\gamma + \sum_{k \in \nn} \gamma_k^+ e(n+1,k) + \sum_{k \in \nn} \gamma_k^- d(n+1,k),
\]
which belongs to the right-hand side of~\eqref{eq:degree-control-step}. This proves the containment ``$\subseteq$'' in~\eqref{eq:degree-control-step}, and the reverse containment follows because the displayed generators are the degrees of elements of $R_{n+1}$. The same normal form also shows that every exponent occurring in an element of $R_{n+1}$ belongs to $V_{\le n+1}$.
\smallskip

Now we construct a nested sequence $(R_n)_{n \ge 0}$ of countable subrings of the monoid domain $\zz[N]$. We proceed inductively. Set
\[
	R_0 := \zz[Q+\nn_0 r].
\]
Since $Q+\nn_0r \subset N$, the ring $R_0$ is a countable subring of $\zz[N]$. Also, $x^r \in R_0$, every exponent occurring in a nonzero element of $R_0$ belongs to $V_{\le 0}$, and
\[
	M_0 := \deg R_0 = Q+\nn_0r.
\]
Having constructed $R_n$ and $M_n := \deg R_n$ for some $n \in \nn_0$, let $(f_{n,k}(x))_{k \ge 1}$ be a sequence of pairwise distinct elements whose underlying set is $R_n \setminus \{0\}$. Define
\[
	X_{n+1} := \big\{ x^{e(n+1,k)} : k \in \nn \big\} \quad \text{and} \quad Y_{n+1} := \big\{ f_{n,k}x^{r-e(n+1,k)} : k \in \nn \big\}.
\]
Since $e(n+1,k), r-e(n+1,k) \in E_{n+1} \subset N$ for every $k \in \nn$, both $X_{n+1}$ and $Y_{n+1}$ are subsets of $\zz[N]$. Now set
\[
	R_{n+1} := R_n[X_{n+1} \cup Y_{n+1}]
\]
and
\[
	D_{n+1} := \{d(n+1,k) : k \in \nn\}, \quad \text{where} \quad d(n+1,k) := r - e(n+1,k) + \deg f_{n,k}.
\]
The degree-control observation gives
\begin{equation} \label{eq:M_n}
	M_{n+1} = M_n + \big\langle e(n+1,k), d(n+1,k) : k \in \nn \big\rangle.
\end{equation}
In particular, $M_{n+1} \subseteq M_n + E_{n+1} \subseteq V_{\le n+1}$. This completes the inductive step, and so we have constructed a nested sequence $(R_n)_{n \ge 0}$ of countable subrings of $\zz[N]$ such that $x^r \in R_n$ and $M_n := \deg R_n \subset V_{\le n}$ for every $n \in \nn_0$. Thus,
\begin{equation} \label{eq:NA domain R}
	R := \bigcup_{n \in \nn_0} R_n
\end{equation}
is a countable subring of the monoid domain $\zz[N]$. After setting $M := \deg R$ and $D := \bigcup_{n \in \nn} D_n$, we see that
\begin{equation} \label{eq:M def:monoid of leading exponents of R}
	M = \bigcup_{n \in \nn_0} M_n = M_0 + \langle A \cup D \rangle = Q + \langle A \cup D \rangle.
\end{equation}
The last equality holds because $1 \in R_0$, and so for some $k \in \nn$ we have $f_{0,k} = 1$, which yields $r=e(1,k)+d(1,k) \in \langle A \cup D \rangle$. Our primary goal here is to prove that $R$ is a nearly atomic domain that is not atomic. First, we need to show that $M$ is a nearly atomic monoid that is not atomic, which we do in the following lemma.

\begin{lemma} \label{lem:aux nearly atomic non-atomic monoid}
	Let $M$ be the monoid defined in~\eqref{eq:M def:monoid of leading exponents of R}, that is, the monoid consisting of all the degrees in $\zz[\rr]$ of the nonzero polynomial expressions in~$R$. Then the following statements hold.
	\begin{enumerate}
		\item $\mathcal{A}(M) = A \cup D$.
		\smallskip
		
		\item $M$ is nearly atomic but not atomic.
	\end{enumerate}
\end{lemma}

\begin{proof}
	(1) The fact that $M$ is reduced follows immediately from the fact that $M$ is a submonoid of $N = Q \oplus E$, which is reduced by Lemma~\ref{lem:M=Q+E=<A U B>}. To verify that $\mathcal{A}(M) \subseteq A \cup D$, first observe that, because $M$ is reduced, every atom of $M$ must belong to any generating set of $M$. Since $M = Q + \langle A \cup D \rangle$, this gives $\mathcal{A}(M) \subseteq Q \cup A \cup D$. Since $Q$ is divisor-closed in $N$, and therefore in $M$, the fact that $Q$ is antimatter ensures that no element of $Q$ can be an atom of $M$. Therefore $\mathcal{A}(M) \subseteq A \cup D$.
	
	For the reverse inclusion, we only need to keep track of the first stage at which certain generators appear. For each $n,k \in \nn$, the element $e(n,k)$ first appears at stage $n$. Similarly,
	\[
		d(n,k) = r - e(n,k) + \deg f_{n-1,k}
	\]
	belongs to $M_n$ but not to $M_{n-1}$: indeed, $\deg f_{n-1,k} \in V_{\le n-1}$, so the coefficient of $e(n,k)$ in the basis expansion of $d(n,k)$ is $-1$, whereas every element of $M_{n-1}$ belongs to $V_{\le n-1}$.
	
	We now show that $A \cup D \subseteq \mathcal{A}(M)$. To verify that $A \subseteq \mathcal{A}(M)$, first observe that no element of the form $d(n,k)+u$, with $u \in M$, can belong to the submonoid $Q \oplus \nn_0 A$: the coefficient of $r$ in $d(n,k)$ is positive, and adding elements of $M$ cannot eliminate all occurrences of generators from $D$. Hence $Q \oplus \nn_0 A$ is divisor-closed in $M$. Since $A = \mathcal{A}(Q \oplus \nn_0 A)$, it follows that $A \subseteq \mathcal{A}(M)$.

	To show that $D \subseteq \mathcal{A}(M)$, fix $n,k \in \nn$ and write $d(n,k) = v+w$ for some $v,w \in M$. Suppose first that neither $v$ nor $w$ belongs to $Q \oplus \nn_0 A$. Then each of $v$ and $w$ has a divisor from $D$, say $d(i,j) \mid_M v$ and $d(i',j') \mid_M w$. Since $d(n,k) \in M_n$, both $v$ and $w$ belong to $M_n$, and so $i,i' \le n$. Projecting onto the $\qq$-span of $A_n$ shows that the only possible negative coefficient at level $n$ on the right must come from exactly one copy of $d(n,k)$. Thus one of $v$ and $w$ has $d(n,k)$ as a divisor, while the other must be~$0$, contradicting the assumption that both lie outside $Q \oplus \nn_0 A$. Hence, after relabeling if necessary, we may assume that $v \in Q \oplus \nn_0 A$.
	
	If $w \ne 0$ and $w \notin Q \oplus \nn_0 A$, then we can take $m,j \in \nn$ such that $d(m,j) \mid_M w$. Thus, we can write
	\begin{equation} \label{eq:D is a set of atoms}
		d(n,k) = v+w = \bigg(q + \sum_{a \in A} c_a a \bigg) + \bigg( d(m,j) + q' + \sum_{a \in A} c'_a a \bigg)
	\end{equation}
	for some $q,q' \in Q$ and two families $\{c_a : a \in A$ and $\{c'_a : a \in A\}$ in $\nn_0^{(A)}$. The preceding stage comparison gives $m=n$. Projecting both sides of~\eqref{eq:D is a set of atoms} onto the $\qq$-span of $A_n$, we see that the unique negative coefficient on the left occurs at $e(n,k)$, while the unique negative coefficient coming from $d(n,j)$ on the right occurs at $e(n,j)$. Hence $j=k$. Substituting $m=n$ and $j=k$ in~\eqref{eq:D is a set of atoms} gives
	\[
		q+q' + \sum_{a \in A} (c_a+c'_a)a = 0.
	\]
	By the linear independence of $\{1\} \cup A$ over $\qq$, we obtain $q=q'=0$ and $c_a=c'_a=0$ for every $a \in A$. Thus, $v=0$ and $w=d(n,k)$. If instead $w \in Q \oplus \nn_0 A$, then the equality $d(n,k)=v+w$ is impossible, because $v+w \in Q \oplus \nn_0 A$ while $d(n,k) \notin Q \oplus \nn_0 A$. Therefore $d(n,k)$ is an atom of $M$, and so $D \subseteq \mathcal{A}(M)$.
	\smallskip
	
	(2) Let us argue now that $M$ is nearly atomic with shift~$r$. Since $M = Q + \langle A \cup D \rangle$, it suffices to prove that $r + Q \subseteq  \langle A \cup D \rangle$. To do so, fix $q \in Q$, and then take $k \in \nn$ such that $f_{0,k} = x^q$. Now observe that
	\[
		r+q = e(1,k) + \big(\deg f_{0,k}(x) + r - e(1,k)\big) = e(1,k) + d(1,k) \in \langle A \cup D \rangle.
	\]
	Thus, $r+M \subseteq \langle A \cup D \rangle$, and so $M$ is nearly atomic. Finally, $M$ is not atomic: if $q \in Q_{>0}$, then $q \notin \langle A \cup D \rangle$ by the linear independence of $\{1,r\} \cup A$ over $\qq$, and therefore $q$ cannot be written as a sum of atoms of~$M$.
\end{proof}

We are in a position to prove that $R$ is a nearly atomic integral domain that is not atomic.

\begin{theorem} \label{thm:NA domain not atomic}
	The integral domain in~\eqref{eq:NA domain R} is nearly atomic but not atomic.
\end{theorem}

\begin{proof}
	Let $R$ be the integral domain introduced in~\eqref{eq:NA domain R}. Observe that $R$ is the subring of the monoid domain $\zz[N]$ generated by the set $\{x^q : q \in Q\} \cup X \cup Y$, where
	\[
		X := \bigcup_{n \in \nn} X_n = \{x^{e(n,k)}:n,k\in\mathbb{N}\} \quad \text{and}  \quad Y := \bigcup_{n \in \nn} Y_n = \{f_{n,k} x^{r-e(n,k)}:n,k\in\mathbb{N} \}.
	\]
	As we have seen in~\eqref{eq:M def:monoid of leading exponents of R}, the equality $M = Q + \langle A \cup D \rangle$ holds. We have proved in Lemma~\ref{lem:aux nearly atomic non-atomic monoid} that $M$ is a reduced nearly atomic monoid with shift $r$ such that $\mathcal{A}(M) = A \cup D$. We verify that $R^\times = \{\pm 1\}$. it is clear that $\{\pm 1\} \subseteq R^\times$ and from the fact that $N$ is reduced, we infer that $R^\times = \{\pm 1\}$: indeed, the fact that $R$ is a subring of the monoid domain $\zz[N]$ guarantees that
	\[
		R^\times \subseteq \zz[N]^\times = \{\pm x^u : u \in \uu(N) \,\} = \{\pm 1\}.
	\]
	We shall also use the following observation. If $g \in R$ is nonzero and $\deg g \in \mathcal{A}(M)$, then $g$ is atomic in $R$. Indeed, any proper factorization of $g$ in $R$ has one factor of degree $0$, and every degree-$0$ element of $R$ belongs to $\zz$; hence repeated proper factorizations of $g$ can only split off nonunit integer factors, and this process terminates by the usual factorization in $\zz$.
	
	We are in a position to argue that $R$ is nearly atomic with shift $x^r$ but not atomic. To prove that $R$ is nearly atomic, take a nonzero $f \in R$, and let $m \in \nn_0$ be an index such that $f \in R_m$. Then $f = f_{m,k}$ for some $k \in \nn$ and, therefore,
	\[
		x^r f = x^{e(m+1,k)} \big(f x^{r-e(m+1,k)}\big).
	\]
	The first factor has degree $e(m+1,k) \in A$, while the second factor has degree $d(m+1,k) \in D$. Since $A \cup D = \mathcal{A}(M)$, the observation in the previous paragraph shows that both factors are atomic in $R$. Hence every nonzero element of $x^rR$ is atomic, and so $R$ is nearly atomic.

	Finally, let us verify that $R$ is not atomic. Observe that $S := \{\pm x^q : q \in Q\}$ is a non-atomic submonoid of the multiplicative monoid $R^*$ because the reduced monoid of $S$ is naturally isomorphic to $Q$. Therefore proving that $R$ is not atomic amounts to arguing that $S$ is a divisor-closed submonoid of $R^*$. To do so, take $g,h \in R^*$ such that $g h \in S$. Then $g$ and $h$ are monomial expressions and their coefficients must belong to $\{\pm 1\}$, whence $g$ and $h$ are monomials of $\zz[N]$ up to sign. Since $\deg \, g + \deg \, h$ belongs to $Q$, the fact that $Q$ is a divisor-closed submonoid of $N = Q \oplus E$ ensures that $\deg \, g, \deg \, h \in Q$, whence $g,h \in S$. Hence $S$ is a divisor-closed non-atomic submonoid of $R^*$, and so the integral domain $R$ is not atomic.
\end{proof}

The construction above takes place inside the monoid algebra $\zz[N]$, but the resulting integral domain is not itself a monoid algebra: it is a recursively generated subring whose generators are chosen so that multiplication by $x^r$ forces atomic factorizations. As a result, the construction leaves open whether the same phenomenon can occur in the class of monoid algebras. Motivated by this and also by the fact that atomicity and factorization in monoid algebras have been the subject of study of several recent papers \cite{BeB22,BGLZ24,GG25}, we pose the following question.

\begin{question} \label{q:monoid domain NA not atomic}
	Does there exist an integral domain $R$ and a torsion-free (cancellative and commutative) monoid $M$ such that the monoid domain $R[M]$ is nearly atomic but not atomic?
\end{question}

\bigskip
%%%%%%%%%%%%%%%%%%%%%%%%%%%%%%%%%%%%%%%%%%%%%%%%%%%%%%%%%%%%%%%%%
%%%%%%%%%%%%%%%%%%%%%%%%%%%%%%%%%%%%%%%%%%%%%%%%%%%%%%%%%%%%%%%%%
\section{A Characterization of the Finite Factorization Property}
\label{sec:a characterization of FFDs via near atomicity}

It is well known that an integral domain has the FF property if and only if it is an atomic IDF domain. This equivalence was first established by Anderson, Anderson, and Zafrullah \cite[Theorem~5.1]{AAZ90} in their seminal paper introducing the FF property (and the BF property), and it was extended to the larger class of monoids two years later by Halter-Koch~\cite{fHK92}, who also proved that for a reduced atomic monoid, the FF property is equivalent to the condition that every element has only finitely many divisors~\cite[Theorem~2]{fHK92}.

Our next goal is to strengthen this characterization by proving that every nearly atomic IDF domain has the FF property. In order to prove this, it is convenient to assume that the multiplicative monoid of the corresponding integral domain is reduced, which we can do by virtue of the following lemma.

\begin{lemma} \label{lem:near atomicity to reduced monoids}
	A monoid $M$ is nearly atomic if and only if its reduced monoid $M/M^\times$ is nearly atomic.
\end{lemma}

\begin{proof}
	Let $\pi \colon M \to M/M^\times$ denote the quotient homomorphism, and observe that $a \in M$ is an atom if and only if $\pi(a)$ is an atom. We claim that $\pi$ also preserves atomic elements.
	\smallskip

	\noindent \textsc{Claim.} An element $b \in M$ is atomic if and only if $\pi(b)$ is atomic in $M/M^\times$.
	\smallskip
	
	\noindent \textsc{Proof of Claim.} Fix $b \in M$. For the direct implication, assume that $b$ is atomic. If $b \in M^\times$ then $\pi(b) = M^\times$, which is the identity element and so an atomic element. We assume therefore that $b \in M \setminus M^\times$ and write $b = a_1 \cdots a_n$ for $n \in \nn$ and $a_1, \dots, a_n \in \mathcal{A}(M)$. Then $\pi(b) = \pi(a_1) \cdots \pi(a_n)$, and so $\pi(b)$ is atomic in $M/M^\times$. Conversely, assume that $\pi(b)$ is atomic in $M/M^\times$. If $\pi(b) = M^\times$ then $b \in M^\times$ and so $b$ is atomic in $M$. We assume, therefore, that $\pi(b) \neq M^\times$ and write $\pi(b) = \pi(a_1) \cdots \pi(a_n)$ for some $n \in \nn$ and $a_1, \dots, a_n \in \mathcal{A}(M)$. Then $\pi(b) = \pi(a_1 \dots a_n)$, and so we can take $u \in M^\times$ such that $b = (ua_1)a_2 \cdots a_n$, which implies that $b$ is atomic in $M$. The claim is now established.
	\smallskip

	%First observe that, for each $a \in M$, the element $a$ is atomic in $M$ if and only if $\pi(a)$ is atomic in $M/M^\times$. Indeed, atoms of $M$ map to atoms of $M/M^\times$, and if $\pi(b)$ is an atom of $M/M^\times$, then $b$ is an atom of $M$: any factorization $b=cd$ in $M$ induces a factorization $\pi(b)=\pi(c)\pi(d)$, so one of $\pi(c)$ and $\pi(d)$ is a unit, whence one of $c$ and $d$ is a unit in $M$.
	
	Now suppose that $M$ is nearly atomic, and let $s \in M$ be a shift such that every element of $sM$ is atomic in $M$. Then every element of $\pi(s)(M/M^\times)$ has the form $\pi(sm)$ for some $m \in M$ and so must be atomic by the established claim. Hence $M/M^\times$ is nearly atomic. Conversely, suppose that $M/M^\times$ is nearly atomic, and choose $s \in M$ such that every element of $\pi(s)(M/M^\times)$ is atomic. For each $m \in M$, the fact that $\pi(sm) = \pi(s)\pi(m) \in \pi(s)(M/M^\times)$ guarantees that $\pi(sm)$ is atomic in $M/M^\times$, which implies that $sm$ is atomic in $M$ by the established claim. Hence $M$ is nearly atomic.
\end{proof}

We are in a position to prove our last main result.

\begin{theorem} \label{thm:a characterization of FFDs via near atomicity}
	For an integral domain $R$, the following conditions are equivalent.
	\begin{enumerate}
		\item[(a)] $R$ is an FFD.
		\smallskip
		
		\item[(b)] $R$ is a nearly atomic IDF domain.
	\end{enumerate}
\end{theorem}

\begin{proof}
	Since the integral domain $R$ is nearly atomic (resp., atomic, an IDF domain, an FFD) if and only if its multiplicative monoid $R^*$ satisfies the corresponding property as a monoid, it suffices to let $M$ denote $R^*$ and prove that $M$ is an FFM if and only if $M$ is a nearly atomic IDF monoid. In addition, it is well known that $M$ is atomic (resp., an IDF monoid, an FFM) if and only if its reduced monoid $M/M^\times$ is atomic (resp., an IDF monoid, an FFM), and it follows from Lemma~\ref{lem:near atomicity to reduced monoids} that $M$ is nearly atomic if and only if its reduced monoid is nearly atomic. Thus, after replacing $M$ by its reduced monoid if necessary, we can assume that $M$ is reduced.
	\smallskip
	
	(a) $\Rightarrow$ (b): Suppose that $M$ is an FFM. By definition, $M$ is atomic and, therefore, $M$ is nearly atomic. In addition, it follows from \cite[Corollary~2]{fHK92} that $M$ is an IDF monoid.
	\smallskip
	
	(b) $\Rightarrow$ (a):
	Assume now that $M$ is a nearly atomic IDF monoid. Since $M$ is nearly atomic, there exists a shift $s \in M$ such that every element of the principal ideal $sM$ is atomic. We first prove the following claim.
	\smallskip
	
	\noindent \textsc{Claim.} The monoid $M$ satisfies the ACCP.
	\smallskip
	
	\noindent \textsc{Proof of Claim.} Suppose, by way of contradiction, that $M$ violates the ACCP. Then there exists an infinite sequence $(a_n)_{n \ge 1}$ of elements of $M$ such that $a_n M \subsetneq a_{n+1}M$ for every $n \in \mathbb N$. For each $n \in \mathbb N$, the fact that $a_1 \in a_1M \subseteq a_nM$ allows us to pick $b_n \in M$ such that $a_1 = a_n b_n$ and then write
	\begin{equation} \label{eq:temp}
		s^2a_1=(sa_n)(sb_n), %\quad
		%s^2a_1 = (sa_n)\left(\frac{sa_1}{a_n}\right),
	\end{equation}
	where $s^2a_1$ and both factors on the right-hand side belong to $sM$ and are therefore atomic. We now consider the submonoid of $M$ generated by the atoms of $M$ that divide $s^2 a_1$:
	\[
		M_1 := \langle a \in \mathcal{A}(M) : s^2a_1M \subseteq aM \rangle.
	\]
	Since $M_1$ is generated by atoms of $M$ (which remain atoms in $M_1$), it follows that $M_1$ is an atomic monoid. In addition, as the inclusion $\mathcal{A}(M_1) \subseteq \mathcal{A}(M)$ holds, for each $b \in M_1$, every atom of $M_1$ that divides $b$ in $M_1$ is an atom of $M$ that divides $b$ in~$M$. Thus, the fact that $M$ is an IDF monoid implies that
	\[
		|\{a \in \mathcal{A}(M_1) : bM_1 \subseteq aM_1 \}| \le |\{a \in \mathcal{A}(M) : bM \subseteq aM \}| < \infty.
	\]
	As a consequence, $M_1$ is an atomic IDF reduced monoid, and this implies that each element of $M_1$ has only finitely many divisors in $M_1$~\cite[Theorem~2]{fHK92}. Since every atomic divisor of $s^2a_1$ in $M$ belongs to $M_1$, it follows from~\eqref{eq:temp} that $sa_n,sb_n \in M_1$ for every $n \in \nn$. By cancellativity, the terms of the sequence $(sa_n)_{n \ge 1}$ are pairwise distinct. As we can write $s^2a_1=(sa_n)(sb_n)$ with both factors in $M_1$, the underlying set of $(sa_n)_{n \ge 1}$ is an infinite set of divisors of $s^2a_1$ in $M_1$. However, this contradicts the fact that each element of $M_1$ has only finitely many divisors. Hence $M$ satisfies the ACCP.
	\smallskip
	
	As $M$ satisfies the ACCP, $M$ must be an atomic monoid. Since $M$ is also an IDF monoid, it follows from the classical characterization of the FF property \cite[Corollary~2]{fHK92} that $M$ is an FFM, which concludes the proof.
\end{proof}

With notation as in Theorem~\ref{thm:a characterization of FFDs via near atomicity}, if one replaces near atomicity by almost atomicity in condition~(b) then the resulting statement is no longer equivalent to condition~(a). We conclude with an example that illustrates this observation.

\begin{example}
	Let $M = (\nn_0 \times \nn_0) \cup (\zz \times \nn_{\ge 2})$, and consider the monoid algebra $R = \rr[X;M]$. Since the Grothendieck group of $M$ is $\zz^2$, we identify the group algebra $\rr[\zz^2]$ with the Laurent polynomial domain $\rr[x^{\pm 1}, y^{\pm 1}]$ in the variables $x$ and $y$, writing $x^my^n$ in place of $X^{(m,n)}$ for every pair $(m,n) \in \zz^2$. Obesrve that, under this identification, $R$ is an intermediate domain of the ring extension $\rr[x,y] \subset \rr[x^{\pm 1}, y^{\pm 1}]$. On the other hand, every polynomial expression $f \in R$ can be multiplied by a monomial $x^k$, for a suitable $k \in \nn_0$, so that $x^k f$ belongs to the UFD $\rr[x,y]$. Let us prove the following claim.
	\smallskip
	
	\noindent \textsc{Claim 1.} $R$ is an IDF domain.
	\smallskip
	
	\noindent \textsc{Proof of Claim 1.} Fix $f \in R$, and choose $a \in \mathbb N_0$ such that $g := x^a f \in \mathbb R[x,y]$. Let $h \in R$ be an irreducible divisor of $f$. If $h \ne x$ then there exists $b \in \mathbb N_0$ such that $x^b h \in \mathbb R[x,y]$ and $x \nmid_{\mathbb R[x,y]} x^b h$. Also,
	\[
		x^b h \mid_{\mathbb R[x,y]} x^{\max\{0, b-a\}} g.
	\]
	Since $x^b h$ is not divisible by~$x$, it must divide $g$. As $g$ has a unique factorization in $\mathbb R[x,y]$, there are only finitely many possibilities for $x^b h$, and each such possibility determines at most one irreducible $h$. Thus, $f$ is divisible by only finitely many irreducibles of $R$. Hence $R$ is an IDF domain, and Claim~1 is established.
	\smallskip
	
	Let $g$ be an irreducible polynomial in~$\mathbb R[x,y]$ for which there exists a largest nonnegative integer $a$ such that $g/x^a \in R$. Assume, in addition, that $g/x^a$ is not a unit of~$R$.
	\smallskip

	\noindent \textsc{Claim 2.} $g/x^a \in \mathcal{A}(R^*)$.
	\smallskip
	
	\noindent \textsc{Proof of Claim 2.} Set $f := g/x^a$. By the maximality of $a$, we see that $x \nmid_R f$. Suppose that $f$ can be factored into two nonunits of~$R$:
	\[
	 	f = \frac{g}{x^a} = \frac{g_1}{x^{k_1}} \cdot \frac{g_2}{x^{k_2}},
	\]
	where $g_1, g_2 \in \mathbb R[x,y]$ and $k_1, k_2 \in \mathbb N_0$. We may assume that $x \nmid_{\mathbb R[x,y]} g_i$ for $i \in \{1,2\}$ since otherwise we can cancel a power of $x$ from $g_i$ and decrease $k_i$. If $a > k_1 + k_2$ then
	\[
		g = g_1g_2x^{a-k_1-k_2},
	\]
	which contradicts the fact that $g$ is irreducible. If $a = k_1 + k_2$ then $g = g_1g_2$. Since $g$ is irreducible, either $g_1$ or $g_2$ is a unit of $\mathbb R[x,y]$. Without loss of generality, assume that $g_1 = 1$. From the fact that
	\[
		\frac{g_1}{x^{k_1}} = \frac{1}{x^{k_1}} \in R,
	\]
	we obtain the equality $k_1 = 0$. Hence $g_1/x^{k_1} = 1$ is a unit of~$R$, a contradiction. Finally, if $a < k_1 + k_2$ then
	\[
		gx^{k_1+k_2-a} = g_1g_2,
	\]
	which is impossible because $x \nmid_{\mathbb R[x,y]} g_i$ for $i \in \{1,2\}$. In all cases, we obtain a contradiction. Hence $f$ is irreducible, and Claim~2 is established.
	\smallskip
	
	We proceed to argue that $R$ is almost atomic. To do so, take any nonzero $f \in R$ and choose a minimal $a \in \mathbb N_0$ such that $g = x^a f \in \mathbb R[x,y]$. Consider the unique factorization
	\[
		g = \prod_{j=1}^\ell g_j
	\]
	of $g$ into (not necessarily distinct) irreducibles $g_1, \dots, g_\ell \in \mathbb R[x,y]$. For each $j \in \ldb 1,\ell \rdb$, let $a_j$ be the largest nonnegative integer such that $g_j/x^{a_j} \in R$. By Claim~2, each nonunit among the elements $g_j/x^{a_j}$ is an irreducible of~$R$. Then
	\[
		x^a f = \prod_{j=1}^\ell g_j = \prod_{j=1}^\ell x^{a_j} \cdot \prod_{j=1}^\ell \frac{g_j}{x^{a_j}}
	\]
	is a factorization of $x^a f$ into irreducibles of~$R$ after expanding the powers $x^{a_j}$ and omitting any unit factors. Hence $f$ is almost atomic and, as $f$ was taken to be an arbitrary element in $R^*$, we conclude that $R$ is an almost atomic domain.
	
	Finally, we show that $R$ is not atomic. Consider any monomial $f$ of $R$ with negative $x$-degree. Any factorization of $f$ in $R$ must be a factorization into monomials. However, the only irreducible monomials of $R$ are $x$ and $y$, and no product of these monomials has negative $x$-degree. Hence $f$ does not factor into irreducibles in $R$, and so $R$ is not atomic. Therefore $R$ is an almost atomic IDF domain that is not atomic, and so~$R$ cannot be an FFD.
	\hfill $\blacksquare$
\end{example}

\bigskip
%%%%%%%%%%%%%%%
%%%%%%%%%%%%%%%
\section*{Acknowledgments}

The collaboration leading to this paper took place as part of CrowdMath, a year-long online math research program hosted
by MIT-PRIMES and the AoPS. The authors thank both MIT-PRIMES and AoPS for making possible this collaboration. The second author kindly acknowledges support from NSF under the award DMS-2213323.

\bigskip
%%%%%%%%%%%%%%%%%%%%%%%%%%%%%%%%%%%%%%%%%
%%%%%%%%%%%%%%%%%%%%%%%%%%%%%%%%%%%%%%%%%
\section*{Conflict of Interest Statement}

On behalf of all authors, the corresponding author states that there is no conflict of interest related to this paper.

\bigskip
%%%%%%%%%%%%%%
%%%%%%%%%%%%%%

\end{document}